\documentclass[reqno]{amsart}
\usepackage{amsthm}
\usepackage{amsmath}
\usepackage[all,ps,cmtip]{xy}
\usepackage{amssymb}
\usepackage{mathrsfs}
\usepackage{enumerate}
\usepackage{enumitem}
\DeclareMathSymbol{A}{\mathalpha}{operators}{`A}%
\DeclareMathSymbol{B}{\mathalpha}{operators}{`B}%
\DeclareMathSymbol{C}{\mathalpha}{operators}{`C}%
\DeclareMathSymbol{D}{\mathalpha}{operators}{`D}%
\DeclareMathSymbol{E}{\mathalpha}{operators}{`E}%
\DeclareMathSymbol{F}{\mathalpha}{operators}{`F}%
\DeclareMathSymbol{G}{\mathalpha}{operators}{`G}%
\DeclareMathSymbol{H}{\mathalpha}{operators}{`H}%
\DeclareMathSymbol{I}{\mathalpha}{operators}{`I}%
\DeclareMathSymbol{J}{\mathalpha}{operators}{`J}%
\DeclareMathSymbol{K}{\mathalpha}{operators}{`K}%
\DeclareMathSymbol{L}{\mathalpha}{operators}{`L}%
\DeclareMathSymbol{M}{\mathalpha}{operators}{`M}%
\DeclareMathSymbol{N}{\mathalpha}{operators}{`N}%
\DeclareMathSymbol{O}{\mathalpha}{operators}{`O}%
\DeclareMathSymbol{P}{\mathalpha}{operators}{`P}%
\DeclareMathSymbol{Q}{\mathalpha}{operators}{`Q}%
\DeclareMathSymbol{R}{\mathalpha}{operators}{`R}%
\DeclareMathSymbol{S}{\mathalpha}{operators}{`S}%
\DeclareMathSymbol{T}{\mathalpha}{operators}{`T}%
\DeclareMathSymbol{U}{\mathalpha}{operators}{`U}%
\DeclareMathSymbol{V}{\mathalpha}{operators}{`V}%
\DeclareMathSymbol{W}{\mathalpha}{operators}{`W}%
\DeclareMathSymbol{X}{\mathalpha}{operators}{`X}%
\DeclareMathSymbol{Y}{\mathalpha}{operators}{`Y}%
\DeclareMathSymbol{Z}{\mathalpha}{operators}{`Z}%
\newcommand{\ideal}[1]{\mathfrak{#1}}
\newcommand{\mrm}[1]{\mathrm{#1}}
\newcommand{\Spec}{\mrm{Spec}}
\newcommand{\Lie}{\mrm{Lie}}
\newcommand{\Hom}{\mrm{Hom}}
\newcommand{\End}{\mrm{End}}
\newcommand{\Aut}{\mrm{Aut}}
\newcommand{\isomto}{\stackrel{\sim}{\longrightarrow}}
\newcommand{\C}{\mathbf{C}}
\newcommand{\Z}{\mathbf{Z}}

\numberwithin{equation}{subsection}

\begin{document}
	\title{Frobenius lifts and elliptic curves with complex multiplication}
	\author{Lance Gurney}
	\swapnumbers
	\newtheorem{theo}[subsection]{Theorem}
	\newtheorem*{theo*}{Theorem}
	\newtheorem{coro}[subsection]{Corollary}
	\newtheorem{lemm}[subsection]{Lemma}
	\newtheorem{prop}[subsection]{Proposition}
	\newtheorem{exmp}[subsection]{Example}
	\theoremstyle{remark}
	\newtheorem{rema}[subsection]{Remark}
	
	%
	\email{lance.gurney@gmail.com}

	\begin{abstract} 
		We give a new characterisation of elliptic curves of Shimura type in terms commuting families of Frobenius lifts and also strengthen an old principal ideal theorem for ray class fields. These two results combined yield the existence of global minimal models for elliptic curves of Shimura type, generalising a result of Gross. Along the way we also prove a handful of small but new results regarding elliptic curves with complex multiplication.
	\end{abstract}

	\maketitle
	\setcounter{tocdepth}{1}
	\tableofcontents
	
	\section*{Introduction} The aim of this paper is to prove several new results on elliptic curves with complex multiplication, most of which are generalisations of previous ones. Let $K$ be an imaginary quadratic field, $L$ an extension of $K$ and $E/L$ an elliptic curve with complex multiplication. In \S\S 1--2 we analyse isogenies between such curves and the associated adelic representations. We then give a simple classification of all such curves over $L$ in terms of these representations. The content of these two sections is most likely well-known (by those who well-know it), but it is difficult to find precise references and so we include it for convenience.
	
	In \S 3 we give a version of the criterion of N\'eron-Ogg-Shafarevich adapted to elliptic curves with complex multiplication. An interesting corollary of this is the following criterion for good reduction, a special case of which is Theorem 2 of \cite{CoatesWiles1977} where it shown for $L=K$ with class number one and $\ideal{f}$ a split prime of $K$:
	
	\begin{theo*} If there exists an ideal $\ideal{f}\subset O_K$ such that the $\ideal{f}$-torsion of $E$ is rational over $L$ and the map $O_K^\times\longrightarrow (O_K/\ideal{f})^\times$ is injective, then $E$ has good reduction everywhere.
	\end{theo*}
	
	In \S 4 we come to the main objects of our study, which are elliptic curves of `Shimura type'. An elliptic curve with complex multiplication $E/L$ is said to be of Shimura type if $L$ is an abelian extension of $K$ and if the torsion of $E$ is rational over the maximal abelian extension of $K$ (it is always so over the maximal abelian extension of $L$). We recall Shimura's Theorem on the existence of such elliptic curves with certain good reduction properties and then show that Shimura's result is sharp using the good reduction criterion of \S 3.
	
	The purpose of \S 5 is to give the following new characterisation of such elliptic curves in terms of commuting families of Frobenius lifts:
	
	\begin{theo*} If $L/K$ is an abelian extension and $E/L$ is an elliptic curve with complex multiplication then $E/L$ is of Shimura type if and only if:
		\begin{enumerate}[label=\textup{(\roman*)}]
			\item for all primes $\ideal{p}$ of $K$, where $E/L$ has good reduction and $L/K$ is unramified, there exists an isogeny \[\psi^\ideal{p}: E\longrightarrow \sigma_\ideal{p}^*(E),\] whose extension to the N\'eron model of $E/L$ reduces modulo $\ideal{p}$ to the $N\ideal{p}$-power relative Frobenius (here $\sigma_{\ideal{p}}\in G(L/K)$ is the Frobenius element at $\ideal{p}$), and
			\item for two prime ideals $\ideal{p}$ and $\ideal{l}$ as in \textup{(i)} the isogenies $\psi^\ideal{p}$ and $\psi^\ideal{l}$ commute in the sense that: \[\sigma_\ideal{l}^*(\psi^\ideal{p})\circ \psi^\ideal{l}=\sigma_\ideal{p}^*(\psi^\ideal{l})\circ \psi^\ideal{p}.\]\end{enumerate}
	\end{theo*}
	
	In \S6 we consider the existence of minimal models of elliptic curves of Shimura type. This questions was (in a sense) already considered by Gross in \cite{Gross82}, where it is shown that if $K$ has prime discriminant and $E/H$ is an elliptic curve of Shimura type then $E$ admits a global minimal model. We give the following generalisation of this:
	
	\begin{theo*} Let $L/K$ be a ray class field with conductor $\ideal{f}$ and let $E/L$ be an elliptic curve of Shimura type. If the $\ideal{f}$-torsion of $E$ is rational over $L$ then $E$ admits a global minimal model away from $\ideal{f}$.
	\end{theo*}
	
	Note that if $\ideal{f}=(1)$ then the $\ideal{f}$-torsion is always rational over the Hilbert class field $H=K((1))$ and so we find that every elliptic curve over Shimura type over $H$ admits a global minimal model (moreover, such curves always exist). 
	
	The main result of \S 6 relies fundamentally on a certain principal ideal theorem which we prove in the appendix. Let $K$ be a number field and let $L/K$ be a wide ray class field. Write $\mathrm{Id}_{L/K}$ for the group of fractional ideals of $O_K$ generated by the primes which are unramified in the extension $L/K$.
	
	\begin{theo*} There exist elements $l(\ideal{a})\in L^\times$, indexed by the ideals $\ideal{a}$ prime to $\ideal{f}$, such that
		\begin{enumerate}[label=\textup{(\roman*)}]
			\item $l(\ideal{a})\cdot O_L=\ideal{a}\cdot O_L$ and
			\item $l(\ideal{a}\ideal{b})=l(\ideal{a})\sigma_{\ideal{a}}(l(\ideal{b}))$
		\end{enumerate} for $\ideal{a}, \ideal{b}$ prime to $\ideal{f}$ where $\sigma_{\ideal{a}}\in G(L/K)$ denotes the `Frobenius element' at $\ideal{a}$.
	\end{theo*} A version of this result was proven by Tanaka in \cite{Tannaka58} (and indeed our proof is heavily based on his and several other classical results from class field theory).

\section*{Acknowledgements} The author would like to thank James Borger for originally suggesting the topic of this article and for many fruitful discussions.

	\section*{Notation} Unless otherwise noted $K$ will denote an imaginary quadratic field, with fixed maximal abelian extension $K^\mathrm{ab}$ and algebraic closure $\overline{K}$.
	
	For a number field $L$, $O_L$ denotes its ring of integers, $\mathrm{Id}_{L}$ the group of fractional ideals, $\widehat{O}_L=O_L\otimes_{\Z}\widehat{\Z}$ the group of finite, integral adeles, and $I_L$ the group of (all) ideles. If $\ideal{P}$ is a prime of $L$ then $L_{\ideal{P}}$ and $O_{L_\ideal{P}}$ denote the $\ideal{P}$-adic completion of $L$ and $O_L$ respectively. Finally, if $\ideal{F}$ is an integral ideal of $L$ and $P\subset \mathrm{Id}_{L}$ is any subset, then $P^{\ideal{F}}$ denotes the subset of ideals of $P$ which are relatively prime to $\ideal{F}$.
	
	If $L/K$ is an abelian extension then we denote by $\mathrm{Id}_{L/K}$ the group of fractional ideals of $K$ generated by the primes which are unramified in $L/K$. We denote by \[\mathrm{Id}_{L/K}\longrightarrow G(L/K): \ideal{a}\mapsto \sigma_{\ideal{a}}\] the unique surjective homomorphism which sends a prime ideal $\ideal{p}\in \mathrm{Id}_{L/K}$ to the unique automorphism lifting $N\ideal{p}$-power Frobenius modulo $\ideal{p}$, we write $P_{L/K}\subset \mathrm{Id}_{L/K}$ for its kernel and for $\ideal{a}\in \mathrm{Id}_{L/K}$ we call $\sigma_{\ideal{a}}$ the Frobenius element at $\ideal{a}$.
	
	The ray class field of conductor $\ideal{f}$ is denoted $K(\ideal{f})$ and when $\ideal{f}=(1)$ we write $H=K((1))$ for the Hilbert class field of $K$. Finally, we denote by \[\theta_K: (\widehat{O}_K\otimes_{O_K}K)^\times \longrightarrow G(K^\mathrm{ab}/K)\] the reciprocity map of class field theory. Note that as $K$ is imaginary quadratic, this homomorphism is surjective with kernel $K^\times$ and its restriction to $\widehat{O}_K^\times\subset (\widehat{O}_K\otimes_{O_K}K)^\times$ induces a surjective map \[\theta_K: \widehat{O}_K^\times\longrightarrow G(K^\mathrm{ab}/H)\] with kernel $O_K^\times.$ We will also use the symbol $\theta_K$ to denote the induced isomorphisms between the corresponding quotient group and the relevant Galois group.

	\section{Isogenies between elliptic curves with complex multiplication}
	
	\subsection{} Let $S$ be an $O_K$-scheme. An elliptic curve $E$ over $S$ is a smooth, proper, geometrically connected $S$-group scheme of relative dimension one. The tangent space at the identity is denoted $\Lie_{E/S}$ (and is a locally free $\mathscr{O}_S$-module of rank one). An elliptic curve with complex multiplication by $O_K$ over $S$ is an elliptic curve $E/S$ equipped with a homomorphism \[O_K\longrightarrow \End_S(E): a\mapsto [a]_E\] such that the induced action of $[a]_E$ on $\Lie_{E/S}$ coincides with the action of $a$ coming from the structure map $S\longrightarrow \Spec(O_K)$. We also call these curves simply `CM elliptic curves'.
	
	\subsection{} We now recall a construction of Serre (Chapter XIII of \cite{CasselsFrohlich67}, also \cite{ASENS_1969_4_2_4_521_0}). Let $E/S$ be a CM elliptic curve. For each $S$-scheme $S'$ the group $\Hom_S(S', E)=E(S')$ is an $O_K$-module and so given a rank one projective $O_K$-module $M$, we may define a functor on the category of $S$-schemes via \[M\otimes_{O_K}E: S'/S\mapsto M\otimes_{O_K} E(S').\]
	
	\begin{prop}[Serre]\label{prop:serre-tensor} The functor $M\otimes_{O_K} E$ is representable by a \textup{CM} elliptic curve over $S$.
	\end{prop}
	\begin{proof} Every rank one projective $O_K$-module $M$ can be generated by a pair of elements and so there exists a surjective homomorphism $O_K^2\longrightarrow M$. Since $M$ is projective, we can split this homomorphism and realise $M$ as the kernel of an idempotent endomorphism $f_M: O_K^2\longrightarrow O_K^2$.
		
		The endomorphism $f_M$ induces an idempotent endomorphism of group schemes $f_M: E^2\longrightarrow E^2$ and an isomorphism (of functors on $S$-schemes): \[M\otimes_{O_K} E\isomto \ker(f_M: E^2\longrightarrow E^2).\] It follows that $M\otimes_{O_K} E$ is representable by a unique group scheme over $S$ with which we now identify it.
		
		As $f_M: E^2\longrightarrow E^2$ is idempotent, $M\otimes_{O_K} E$ is a direct factor of the smooth, proper and geometrically connected group scheme $E^2$. Hence, $M\otimes_{O_K} E$ is itself smooth, proper and geometrically connected.
		
		Finally, the additivity of functor $\Lie$ applied to the left split exact sequence \[0\longrightarrow M\otimes_{O_K} E\longrightarrow E^2\longrightarrow E^2\] induces a canonical isomorphism \[\Lie_{M\otimes_{O_K} E/S}\isomto M\otimes_{O_K}\Lie_{E/S}\] from which it follows that $M\otimes_{O_K} E$ is of both of relative dimension one and a CM elliptic curve, i.e. the induced action of $O_K$ on $\Lie_{M\otimes_{O_K} E/S}$ is via the structure map $S\longrightarrow \Spec(O_K)$.
	\end{proof}
	
	We identify the functor $M\otimes_{O_K}E$ with the representing CM elliptic curve. 
	\begin{rema} We also make a few remarks about this construction.
		\begin{enumerate}[label=\textup{(\roman*)}]
			\item In the proof of (\ref{prop:serre-tensor}) it is shown more generally that if $G/S$ is any group scheme equipped with an action of $O_K$ and $M$ is a rank one projective $O_K$-module then the functor $M\otimes_{O_K} G$ is again representable by a group scheme over $S$ and inherits any properties possessed by direct factors of the group scheme $G^2$, e.g. affine, flat, finite locally free, \'etale and so on. We will use this from time to time when $G$ is a torsion sub-group of $E$ or when $G$ is the N\'eron model of an elliptic curve.
			
			\item The additivity of the Hom functor also shows that, for a pair of CM elliptic curves $E$ and $E'$ over $S$ and a rank one projective $O_K$-module $M$, the natural map \[M\otimes_{O_K}\Hom^{O_K}_S(E, E')\isomto \Hom_S^{O_K}(E, M\otimes_{O_K} E')\] is bijective.
		\end{enumerate}
	\end{rema}

	\subsection{} We now apply Serre's construction in the special case where $M=\ideal{a}^{-1}$ for a non-zero integral ideal $\ideal{a}\subset O_K$ to obtain a CM elliptic curve $\ideal{a}^{-1}\otimes_{O_K} E$.  There is a natural homomorphism \[i_\ideal{a}: E\longrightarrow \ideal{a}^{-1}\otimes_{O_K} E\] induced by the inclusion $O_K\longrightarrow \ideal{a}^{-1}$ whose kernel we denote by $E[\ideal{a}]$. The sub-group scheme $E[\ideal{a}]$ is the $\ideal{a}$-torsion of $E$: the $S'$-valued points of $E[\ideal{a}]$ is \[E[\ideal{a}](S')=\{x\in E(S'): [a]_E(x)=0 \text{ for all } a\in \ideal{a}\}.\] If $\ideal{a}=(a)$ is principal then $E[\ideal{a}]=\ker([a]_E).$
-	
	\begin{prop} The homomorphism $i_\ideal{a}: E\longrightarrow \ideal{a}^{-1}\otimes_{O_K}E$ is finite locally free of degree $N\ideal{a}$ and is \'etale if and only if $\ideal{a}$ is invertible on $S$.
	\end{prop}
	\begin{proof} The homomorphism $i_\ideal{a}$ is \'etale if and only if the induced map \[\Lie_{E/S}\longrightarrow \Lie_{\ideal{a}^{-1}\otimes_{O_K} E}=\ideal{a}^{-1}\otimes_{O_K}\Lie_{E/S}\] is an isomorphism. This map is induced by the inclusion $O_K\longrightarrow \ideal{a}^{-1}$ and is therefore an isomorphism if and only if $\ideal{a}$ is invertible on $S$.
		
		Regarding the degree of $i_\ideal{a}$, by rigidity we may decompose $S$ into a disjoint union of schemes over which $i_\ideal{a}$ is either the zero morphism or finite locally free of constant degree. As $\ideal{a}$ is non-zero $i_\ideal{a}$ cannot be the zero morphism and is therefore finite locally free of constant degree. This degree is one if and only if $\ideal{a}=O_K$ in which case the claim is clear. Therefore, we may assume that $i_\ideal{a}$ is finite locally free of constant degree greater than one.
		
		Given another (non-zero) ideal $\ideal{b}$ of $O_K$, the exactness of the functor $\ideal{b}^{-1}\otimes_{O_K}-$ implies that the kernel of \[\ideal{b}^{-1}\otimes_{O_K} i_\ideal{a}: \ideal{b}^{-1}\otimes_{O_K} E\longrightarrow \ideal{b}^{-1}\otimes_{O_K} \ideal{a}^{-1}\otimes_{O_K} E=(\ideal{a}\ideal{b})^{-1}\otimes_{O_K} E\] is equal to $\ideal{b}^{-1}\otimes_{O_K} E[\ideal{a}]$. Choosing an isomorphism $\ideal{b}^{-1}\otimes_{O_K}(O_K/\ideal{a})\isomto O_K/\ideal{a}$ we also obtain an isomorphism \[\ideal{b}^{-1}\otimes_{O_K}E[\ideal{a}]\isomto E[\ideal{a}]\] and it follows that $\deg(i_\ideal{a}\otimes_{O_K}\ideal{b}^{-1})=\deg(i_\ideal{a})$ and so \[\deg(i_{\ideal{a}\ideal{b}})=\deg((i_\ideal{b}\otimes_{O_K}\ideal{b}^{-1})\circ i_\ideal{b})=\deg(i_\ideal{a})\deg(i_\ideal{b}).\] As $N{\ideal{a}\ideal{b}}=N\ideal{a} N\ideal{b}$ and $\deg(i_{\ideal{a}\ideal{b}})=\deg(i_\ideal{a})\deg(i_\ideal{b})$ it is enough to show that $i_\ideal{p}$ has degree $N\ideal{p}$ whenever $\ideal{p}$ is a non-zero prime ideal. If $\overline{\ideal{p}}$ denotes the complex conjugate of $\ideal{p}$ then \[\deg(i_\ideal{p})\deg(i_{\overline{\ideal{p}}})=\deg(i_{\ideal{p}\overline{\ideal{p}}})=\deg([N\ideal{p}]_E)=N\ideal{p}^2.\] Therefore, if $\ideal{p}=\overline{\ideal{p}}$ then $\deg(i_\ideal{p})^2=N\ideal{p}^2$ and we must have $\deg(i_\ideal{p})=N\ideal{p}$. On the other hand if $\ideal{p}\neq \overline{\ideal{p}}$ then $N\ideal{p}$ is prime and as both $\deg(i_\ideal{p})$ and $\deg(i_{\overline{\ideal{p}}})$ are greater than one we must also have $\deg(i_\ideal{p})=N\ideal{p}$.
	\end{proof}
	
	\begin{prop}\label{prop:cm-subgroups} Let $L/K$ be a finite extension, let $S$ be either $\Spec(L)$ or an open subset of $\Spec(O_L)$ and let $E/S$ be a \textup{CM} elliptic curve. The only finite locally free sub-group schemes of $E$ which are stable under the action of $O_K$ are those of the form $E[\ideal{a}]$ for $\ideal{a}\subset O_K$.
	\end{prop}
	\begin{proof} If $S=\Spec(L)$ then this follows immediately after noting that $E_\mathrm{tors}$ is an ind-finite locally free scheme over $\Spec(L)$ which, after base change to an algebraic closure of $L$, is isomorphic as an $O_K$-module group scheme to the constant $O_K$-module group scheme associated to $K/O_K$ whose only finite $O_K$-sub-modules are given by $\ideal{a}$-torsion $\ideal{a}^{-1}/K\subset K/O_K$ for $\ideal{a}\subset O_K.$
		
		If $S\subset \Spec(O_L)$ is an open sub-scheme and $C\subset E$ a finite locally free subgroup scheme stable under $O_K$ then $C\times_{S}\Spec(L)=(E\times_{S}\Spec(L))[\ideal{a}]$ for some ideal $\ideal{a}$ by the above. Therefore, the degree of $C$ is equal to the degree of $E[\ideal{a}]$ and by Corollary 1.3.5 of \cite{KatzMazur85} there is a unique maximal closed sub-scheme $Z\subset S$ over which $C$ and $E[\ideal{a}]$ are equal. Since $C$ and $E[\ideal{a}]$ are equal over the generic fibre $\Spec(L)\longrightarrow S$ and $Z\subset S$ is closed it follows that $Z=S$ and hence that $C=E[\ideal{a}].$
	\end{proof}
	
	\section{Classification of elliptic curves with complex multiplication}
	
	For details regarding the content of this section see Chapter 1 of \cite{Serre} and Chapter 2 of \cite{Gross1980}.
	
	\subsection{} Let $K\subset L\subset \overline{K}$ be a finite extension with maximal abelian extension $L^\mathrm{ab}\subset \overline{K}$ and let $E/L$ be a CM elliptic curve. The $O_K$-module $E(\overline{K})_\mathrm{tors}$ is isomorphic to $K/O_K$ and $G(\overline{K}/L)$ acts on this module via a character \[\rho_{E/L}: G(L^\mrm{ab}/L)\longrightarrow \widehat{O}_K^\times=\mathrm{Aut}_{O_K}(K/O_K).\] It is customary to classify CM elliptic curves via their Hecke characters, however it is conceptually simpler to instead use their ad\`elic representations $\rho_{E/L}$ directly, keeping in mind that those which appear as such satisfy a certain special property (see (\ref{eqn:commutative}) below). Indeed, this property is exactly what allows one to convert $\rho_{E/L}$ into the associated algebraic Hecke character $\psi_{E/L}$ as we know explain.
	
	Write \[\widetilde{N}_{L/K}: I_L\longrightarrow (\widehat{O}_K\otimes_{O_K} K)^\times\] for the composition of the norm $N_{L/K}: I_L\longrightarrow I_K$ with the projection $I_K\longrightarrow (\widehat{O}_K\otimes_{O_K} K)^\times$ which forgets the archimedian factor. The algebraic Hecke character (see Theorem 10 of \cite{SerreTate68} for an alternative definition) associated to $E/L$ is \[\psi_{E/L}:=\rho_{E/L}^{-1}\cdot \widetilde{N}_{L/K}: I_L\longrightarrow K^\times\subset (\widehat{O}_K\otimes_{O_K}K)^\times.\] A priori $\psi_{E/L}$ takes values in $(\widehat{O}_K\otimes_{O_K}K)^\times$, however it can be shown to take values in $K^\times$ (as follows from Theorem 11 of \cite{SerreTate68}, for example) which is equivalent to the ad\`elic representation $\rho_{E/L}$ having the property that the following diagram commutes: \begin{equation}\begin{gathered}\xymatrix{G(L^\mrm{ab}/L)\ar[r]^-{\rho_{E/L}} \ar[dr]_{\mathrm{res}}&\widehat{O}_K^\times\ar[d]^{\theta_K}\\ & G(K^\mrm{ab}/K)
	}\label{eqn:commutative}\end{gathered}\end{equation} where the diagonal arrow is the restriction map. This implies that the extension $K\subset L$ contains the Hilbert class field $K\subset H\subset \overline{K}$. Moreover, there always exist CM elliptic curves defined over $H$ (see Chapter I of \cite{Serre}).
	
	\subsection{} Now fix an embedding $K\subset \overline{K}\longrightarrow \C$ and let $E/\C$ be a CM elliptic curve. By GAGA the functor $E\mapsto E^\mrm{an}$, sending $E$ to its analytification, is an equivalence of categories between CM elliptic curves over $\C$ and complex tori of dimension one together with an action of $O_K$, which acts through the inclusion $O_K\subset \C$ on the tangent space at the identity.
	
	The exponential map \[\Lie_{E^\mrm{an}}\longrightarrow E^\mrm{an}\] is holomorphic, surjective and $O_K$-linear, its kernel $T_{O_K}(E)$ is a rank one projective $O_K$-module and the functor $E\mapsto T_{O_K}(E)$ from the category of CM elliptic curves over $\Spec(\C)$ to the category of rank one projective $O_K$-modules is an equivalence.
	
	We denote by $CL_K$ the class group of $K$ and we denote by $[M]\in CL_K$ the class of a rank one projective $O_K$-module $M$. If $L\subset \overline{K}$ is a finite extension of $K$ and $E/L$ is a CM elliptic curve we write \[c_{E/L}:=[T_{O_K}(E_{\C}^\mrm{an})]\in \mrm{CL}_{K}\] where $E_\C=E\times_{\Spec(L)}\Spec(\C)$.
	
	\subsection{} We record the following properties of $\rho_{E/L}$ and $c_{E/L}$:
	\begin{enumerate}[label=\textup{(\roman*)}]
		\item If $\chi: G(L^\mrm{ab}/L)\longrightarrow \Aut_L^{O_K}(E)=O_K^\times$ is a character and $E^\chi$ is the twist of $E$ by $\chi$ then \[\rho_{E^\chi/L}=\chi\cdot \rho_{E/L} \quad \text{ and } \quad c_{E^\chi/L}=c_{E/L}.\]
		\item If $\sigma\in G(L/K)$ then \[\rho_{\sigma^*(E)/L}=\rho_{E/L}^\sigma \quad \text{ and } \quad c_{\sigma^*(E)/L}=[\ideal{a}]^{-1}c_{E/L}\] where \[\rho_{E/L}^\sigma(-)=\rho_{E/L}(\widetilde{\sigma}\circ - \circ \widetilde{\sigma}^{-1})\] for any extension of $\sigma\in G(L/K)$ to $\widetilde{\sigma}\in G(L^\mrm{ab}/K)$ and $\sigma|_H=\sigma_{\ideal{a}}$.
		\item If $\ideal{a}$ is a fractional ideal of $K$ then \[\rho_{\ideal{a}\otimes_{O_K} E/L}=\rho_{E/L} \quad \text{ and } \quad c_{\ideal{a}\otimes_{O_K} E/L}=[\ideal{a}]c_{E/L}.\]
	\end{enumerate}
	
	\begin{prop}\label{theo:classification} The assignment \[E/L\mapsto (\rho_{E/L}, c_{E/L})\] from isomorphism classes of \textup{CM} elliptic curves over $L$ to the set of pairs $(\rho, c)$ where
	\begin{enumerate}[label=\textup{(\roman*)}]
				\item $\rho: G(L^\mrm{ab}/L)\longrightarrow \widehat{O}_K^\times$ is a continuous character such that \[\xymatrix{G(L^\mrm{ab}/L)\ar[r]^-{\rho}\ar[dr]_{\mathrm{res}}&\widehat{O}_K^\times\ar[d]^{\theta_K}\\ & G(K^\mrm{ab}/K)
		}\] commutes, and
	\item $c\in CL_K$
	\end{enumerate}
	 is bijective. Moreover, if $E$ and $E'$ are a pair of \textup{CM} elliptic curves over $L$ then $E$ and $E'$ are isogenous if and only if $\rho_{E/L}=\rho_{E'/L}$.
	\end{prop}
	\begin{proof} Let $(\rho, c)$ be a pair as above. The fact that (\ref{eqn:commutative}) commutes implies that the image of $G(L^\mrm{ab}/L)$ in $G(K^\mrm{ab}/K)$ (via restriction) lands in the subgroup $G(K^\mrm{ab}/H)$ so that $H\subset L$. As there exists a CM elliptic curve $H$, base changing we obtain a CM elliptic curve $E/L$.
		
	If $(\rho_{E/L}, c_E)$ is the pair corresponding to $E/L$ then there exists a character $\chi: G(L^\mrm{ab}/L)\longrightarrow O_K^\times$ and a fractional ideal $\ideal{a}$ of $K$ such that \[(\rho, c)=(\chi\rho_{E/L}, [\ideal{a}]c).\] But \[(\rho_{\ideal{a}\otimes_{O_K}E^\chi/L}, c_{\ideal{a}\otimes_{O_K}E^\chi/L})=(\chi\rho_{E/L}, [\ideal{a}]c)\] and so the map in question is surjective.
		
	On the other hand, let $E$ and $E'$ be a pair of CM elliptic curves over $L$ such that \[(\rho_{E/L}, c_E)=(\rho_{E'/L}, c_{E'}).\] The equality $c_{E/L}=c_{E'/L}$ implies that $E_{\C}$ and $E'_{\C}$ are isomorphic, which in turn implies that $E_{\overline{K}}$ and $E'_{\overline{K}}$ are isomorphic and thus there exists a character $\chi: G(L^\mrm{ab}/L)\longrightarrow O_K^\times$ such that $E'$ and $E^\chi$ are isomorphic. We then find \[\rho_{E/L}=\rho_{E'/L}=\chi\cdot \rho_{E/L}\] which is possible if and only if $\chi$ is trivial. Therefore $E'=E^\chi$ is isomorphic to $E$ and the map in question is bijective.
		
	For the last statement let $E$ and $E'$ be a pair of CM elliptic curves over $L$. Then $E$ and $E'$ are isogenous (over $L$) if and only if $E'=\ideal{a}^{-1}\otimes_{O_K} E$ for some integral ideal $\ideal{a}$ of $O_K$ as any isogeny $f$ has $\ker(f)=E[\ideal{a}]$ for some $\ideal{a} \subset O_K$.
	
	It is clear that $\rho_{E/L}=\rho_{\ideal{a}\otimes_{O_K}E/L}$. Conversely, if $\rho_{E/L}=\rho_{E'/L}$ it follows that $E'$ is isomorphic to $M\otimes_{O_K}E$ for some rank one projective $O_K$-module $M$ (by the bijectivity already shown) and choosing any non-zero element $m\in M$ we obtain an isogeny \[E\longrightarrow E'=M\otimes_{O_K}E: x\mapsto m\otimes x.\]
	\end{proof}
	
	\section{Good reduction}
	
	Since it doesn't seem to appear elsewhere, let us give the following version of the criterion of N\'eron-Ogg-Shafarevich adapted to CM elliptic curves (see \cite{SerreTate68} for the original).
	
	\begin{theo}\label{prop:interia-groups} Let $L/K$ be a finite extension, $E/L$ be a \textup{CM} elliptic curve, $\ideal{P}\subset O_L$ a prime ideal lying over the prime $\ideal{p}\subset O_K$ and let $I_\ideal{P}\subset G(L^\mrm{ab}/L)$ be the inertia subgroup. Then \[\rho_{E/L}(I_\ideal{P})\subset O_K^\times\cdot O_{K_\ideal{p}}^\times\subset \widehat{O}_K^\times\] and $E/L$ has good reduction at $\ideal{P}$ if and only if $\rho_{E/L}(I_\ideal{P})\subset O_{K_\ideal{p}}^\times$.
	\end{theo}
	\begin{proof} The image of $I_\ideal{P}$ under the restriction map $G(L^\mrm{ab}/L)\longrightarrow G(K^\mrm{ab}/K)$ is contained in the inertia group $I_\ideal{p}\subset G(K^\mrm{ab}/K)$ which in turn is equal to the image of the sub-group $O_{K_\ideal{p}}^\times\subset \widehat{O}_K^\times/O_K^\times$ under the map $\theta_K:\widehat{O}_K^\times/O_K^\times\longrightarrow G(K^\mrm{ab}/K)$. This observation combined with the fact that the diagram \[\xymatrix{G(L^\mrm{ab}/L)\ar[r]^-{\rho_{E/L}}\ar[dr]_{\mathrm{res}}&\widehat{O}_K^\times\ar[d]^{\theta_K}\\ & G(K^\mrm{ab}/K)}\] commutes shows that $\rho_{E/L}(I_\ideal{P})\subset O_K^\times \cdot O_{K_\ideal{p}}^\times\subset \widehat{O}_K^\times$.

	For the claim regarding good reduction, let $\ell$ be a rational prime such that $\ell\cdot O_K$ is prime to $\ideal{p}$. By the usual criterion of N\'eron-Ogg-Shafarevich, $E/L$ has good reduction at $\ideal{P}$ if and only if the action of $I_\ideal{P}$ on $E[\ell^\infty](L^\mrm{ab})$ is trivial and this action is trivial if and only if the image of $\rho_{E/L}(I_\ideal{P})\subset \widehat{O}_K^\times$ along the projection $\widehat{O}_K^\times\longrightarrow (O_K\otimes_{\Z}\Z_\ell)^\times$ is trivial.
	
	Now, as the intersection of $O_K^\times$ and $O_{K_\ideal{p}}^\times$ inside $\widehat{O}_K^\times$ is trivial, we have $O_K^\times\cdot O_{K_\ideal{p}}^\times=O_K\times O_{K_\ideal{p}}^\times$ and so the restriction of $\rho_{E/L}$ to $I_\ideal{P}$ is a product of two characters \[\rho_{E/L}|_{I_\ideal{P}}=\alpha\cdot\lambda: I_\ideal{P}\to O_K^\times\times O_{K_\ideal{p}}^\times = O_K^\times\cdot O_{K_\ideal{p}}^\times\subset \widehat{O}_K^\times.\]
	
	As $\ideal{p}$ and $\ell\cdot O_K$ are coprime, the image of $\rho_{E/L}(I_\ideal{P})$ in $(O_K\otimes_{\Z}\Z_\ell)^\times$ coincides with the image of $\alpha(I_\ideal{P})$, and as the map $O_K^\times \longrightarrow (O_K\otimes_{\Z}\Z_\ell)^\times$ is injective, it follows that the image of $\alpha(I_\ideal{P})$ is trivial if and only if $\alpha(I_{\ideal{P}})$ is itself trivial. This in turn is true if and only if $\rho_{E/L}(I_\ideal{P})\subset O_{K_\ideal{p}}^\times.$
	\end{proof}
	
	We also obtain the following useful corollary, a special case of which is Theorem 2 of \cite{CoatesWiles1977}:
	
	\begin{coro}\label{coro:good-red-const} Let $L/K$ be a finite extension,  let $E/L$ be a \textup{CM} elliptic curve and let $\ideal{f}\subset O_K$ be an ideal such that the map $O_K^\times\longrightarrow (O_K/\ideal{f})^\times$ is injective. If $E[\ideal{f}]$ is rational over $L$ then $E$ has good reduction everywhere.
	\end{coro}
	\begin{proof} If $\ideal{P}$ is an prime ideal of $O_L$ lying over the prime ideal $\ideal{p}$ of $O_K$, and $I_\ideal{P}\subset G(L^\mrm{ab}/L)$ denotes the inertia group at a prime $\ideal{P}$ of $L$ then by (\ref{prop:interia-groups}) $E/L$ has good reduction at $\ideal{P}$ if and only if the sub-group $\rho_{E/L}(I_\ideal{P})\subset O_K^\times\cdot O_{K_\ideal{p}}^\times\subset  \widehat{O}_K^\times$ is contained in $O_{K_\ideal{p}}^\times\subset \widehat{O}_K^\times.$ Since $O_K^\times$ and $O_{K_\ideal{p}}^\times$ have trivial intersection (as subgroups of $\widehat{O}_K^\times$) and the map $O_K^\times\longrightarrow (O_K/\ideal{f})^\times$ is injective, this is equivalent to the image of $\rho_{E/L}(I_{\ideal{P}})$ along $\widehat{O}_K^\times\longrightarrow (O_K/\ideal{f})^\times$ being trivial. However, $E[\ideal{f}]$ is constant so that $G(L^\mrm{ab}/L)$ acts trivially on $E[\ideal{f}](\overline{K})$ and so the image of $\rho_{E/L}(G(L^\mrm{ab}/L))\subset \widehat{O}_K^\times$ (and a fortiori that of $\rho_{E/L}(I_\ideal{P})$) is trivial in $(O_K/\ideal{f})^\times$.
	\end{proof}
	
	\section{Elliptic curves of Shimura type}
	
	\subsection{} Let $L$ be an abelian extension of $K$. A CM elliptic curve $E/L$ is said to be of Shimura type if the action of $G(\overline{K}/L)$ on $E(\overline{K})_\mathrm{tors}$ factors through $G(K^\mrm{ab}/L)$. Note that if $K$ has class number one then every CM elliptic curve over $K$ is of Shimura type and such curves always exist. More generally, we have the following result of Shimura:
	
	\begin{theo}[Shimura]\label{prop:shimura-lambda-curves-over-hilbert} There exist infinitely many prime ideals $\ideal{p}$ of $K$ with the property that there exists a \textup{CM} elliptic curve $E/H$ of Shimura type with good reduction at all prime ideals of $H$ prime to $\ideal{p}$.
	\end{theo}
	\begin{proof} By Proposition 7, \S 5 of \cite{Shimura71} there exists infinitely many primes $\ideal{p}$ of $K$ with $N\ideal{p}=p$ a rational prime, $N\ideal{p}=1\bmod w$ and $(N\ideal{p}-1)/w$ prime to $w$, where $w=\# O_K^\times$. Given such a prime $\ideal{p}$ it follows that the reduction map \[O_K^\times\longrightarrow (O_K/\ideal{p})^\times\] is the inclusion of a direct factor. Therefore, we may define a retraction $\alpha: \widehat{O}_K^\times\longrightarrow O_K^\times$ of the inclusion $O_K^\times\to \widehat{O}_K^\times$ by \[\widehat{O}_K^\times\longrightarrow (O_K/\ideal{p})^\times\longrightarrow O_K^\times\] where the first map is the quotient map and the second is a retraction of $O_K^\times \subset (O_K/\ideal{p})^\times$.
		
	We now define a character $\rho: G(\overline{K}/H)\longrightarrow \widehat{O}_K^\times$ by \[G(\overline{K}/H)\longrightarrow G(K^\mrm{ab}/H) \isomto \widehat{O}_K^\times/O_K^\times\longrightarrow \widehat{O}_K^\times\] where the last map sends the class of $s\in \widehat{O}_K^\times$ to $s^{-1}\alpha(s)$. The character $\rho$ satisfies the conditions of (\ref{theo:classification}) so that there exists a (not necessarily unique) CM elliptic curve $E/H$ with $\rho_{E/H}=\rho$. Moreover, the action of $G(\overline{K}/H)$ on $E(\overline{K})$ factors through $G(K^\mrm{ab}/H)$ so that $E/H$ is a CM elliptic curve of Shimura type and $\rho=\rho_{E/H}$ is a character $G(K^\mrm{ab}/H)\longrightarrow \widehat{O}_K^\times.$
		
	Finally, for any prime $\ideal{L}$ of $O_H$ lying over a prime $\ideal{l}\neq \ideal{p}$ of $O_K$, the image of $I_{\ideal{L}}\subset G(H^\mrm{ab}/H)$ along \[G(H^\mrm{ab}/H)\longrightarrow G(K^\mrm{ab}/H)\longrightarrow \widehat{O}_K^\times\] is equal to the image of $I_{\ideal{l}}\subset G(K^\mrm{ab}/H)$ under $\rho$.
		
	By (\ref{prop:interia-groups}) $E/H$ has good reduction at the prime $\ideal{L}$ if and only if the composition \[O_{K_\ideal{l}}^\times\longrightarrow \widehat{O}_K^\times/O_K^\times\stackrel{\sim}{\longleftarrow} G(K^\mrm{ab}/H) \stackrel{\rho}{\longrightarrow}\widehat{O}_K^\times \] has image in $O_{K_\ideal{l}}^\times$. As $\ideal{l}\neq \ideal{p}$, this composition is just the inclusion $O_{K_\ideal{l}}^\times\subset \widehat{O}_K^\times$ by the definition of $\rho$ and hence $E/H$ has good reduction at $\ideal{L}$.
	\end{proof}
	
	It is easy to show that the set of primes of $L$ where a CM elliptic curve $E/L$ of Shimura type has good reduction is stable under $G_{L/K}$ and is therefore equal to the set of primes which are prime to some ideal $\ideal{a}$ of $K$. Thus the following shows that (\ref{prop:shimura-lambda-curves-over-hilbert}) is sharp:
	
	\begin{prop}\label{prop:no-good-reduction-shimura} There does not exist a \textup{CM} elliptic curve $E/H$ of Shimura type with good reduction everywhere.
	\end{prop}
	\begin{proof} Let $E/H$ be a CM elliptic curve of Shimura type. The defining property of such curves implies that the character $\rho_{E/H}: G(H^\mrm{ab}/H)\longrightarrow \widehat{O}_K^\times$ factors through the restriction map $G(H^\mrm{ab}/H)\longrightarrow G(K^\mrm{ab}/H)$ and we shall use the same symbol for the induced map.
		
	Composing the reciprocal $\rho_{E/H}$ with the isomorphism $\theta_K:\widehat{O}_K^\times/O_K^\times\isomto G(K^\mrm{ab}/H)$ we obtain a homomorphism \[\eta:\widehat{O}_K^\times/O_K^\times\longrightarrow \widehat{O}_K^\times\] which is a section of the quotient map $\widehat{O}_K^\times\longrightarrow \widehat{O}_K^\times/O_K^\times$.
		
	The elliptic curve $E/H$ has good reduction at all places of $H$ lying above a prime $\ideal{p}$ of $O_K$ if and only if the composition \[O_{K_\ideal{p}}^\times\subset \widehat{O}_K^\times/O_K^\times\stackrel{\eta}{\longrightarrow} \widehat{O}_K^\times\] coincides with the inclusion $O_{K_\ideal{p}}^\times\subset \widehat{O}_K^\times$. Therefore, $E/H$ has good reduction everywhere if and only if the composition \begin{equation}\label{eqn:f-comp-quot}\widehat{O}_K^\times\longrightarrow \widehat{O}_K^\times/O_K^\times \stackrel{\eta}{\longrightarrow} \widehat{O}_K^\times\end{equation} is equal to the identity on the sub-group of $\widehat{O}_K^\times$ generated by the sub-groups $O_{K_\ideal{p}}^\times\subset \widehat{O}_K^\times$ for all primes $\ideal{p}$ of $O_K$. However, this sub-group is dense and $\eta$ is continuous so that (\ref{eqn:f-comp-quot}) itself must be the identity, which is clearly impossible.
	\end{proof}

\begin{rema} In contrast to (\ref{prop:no-good-reduction-shimura}) above, there may exist CM elliptic curves over $H$ with good reduction everywhere. Indeed, Rohrlich has shown \cite{Rohrlich1982} that this is the case precisely when the discriminant of $K$ is divisible by at least two primes congruent to $3\bmod 4.$
\end{rema}
	
\section{Lifts of the Frobenius}
	
\subsection{}\label{subsec:cm-over-a-field-shimura-integral} We now fix some notation. Let $L/K$ be an abelian extension and let $E/L$ be a CM elliptic curve (not necessarily of Shimura type). Let $\ideal{g}\subset O_K$ be an ideal with the property that $S=\Spec(O_L[\ideal{g}^{-1}])$ is unramified over $\Spec(O_K)$ and $E$ has good reduction over $S$, so that the N\'eron model $\mathscr{E}/S$ of $E/L$ is a CM elliptic curve over $S$.
	
	We write $\mrm{Id}_K^\ideal{g}$ for the set of ideals of $O_K$ prime to $\ideal{g}$ and for a prime $\ideal{p}\in \mrm{Id}_K^\ideal{g}$ we write $S_\ideal{p}=S\times_{\Spec(O_K)}\Spec(O_K/\ideal{p})$, $\mathscr{E}_\ideal{p}=\mathscr{E}\times_S S_\ideal{p}$ and $\sigma_\ideal{p}: S\longrightarrow S$ for the Frobenius element at $\ideal{p}$.
	
	\begin{lemm} For each $\ideal{p}\in \mrm{Id}_K^\ideal{g}$, there is at most one homomorphism \[\psi^{\ideal{p}}:\mathscr{E}\longrightarrow \sigma_\ideal{p}^*(\mathscr{E})\] lifting the $N\ideal{p}$-power relative Frobenius map of $\mathscr{E}_\ideal{p}$ and if such a map exists its kernel is equal to $\mathscr{E}[\ideal{p}].$
	\end{lemm}
	\begin{proof} By rigidity the difference of two such homomorphisms is equal to the zero map on some open and closed sub-scheme of $S$, the only choices of which are $S$ and $\emptyset$. Therefore, as any two such homomorphisms must agree on the non-empty sub-scheme $S_\ideal{p}\subset S$, they must agree everywhere.
		
		By (\ref{prop:cm-subgroups}) we have $\ker(\psi^\ideal{p})=\mathscr{E}[\ideal{a}]$ for some integral ideal $\ideal{a}$ of $O_K$. Since $S$ is connected and $\psi^\ideal{p}$ lifts the $N\ideal{p}$-power relative Frobenius it must have degree $N\ideal{p}$. If $\ideal{p}=\overline{\ideal{p}}$ then $\mathscr{E}[\ideal{p}]$ is the unique sub-group scheme of $\mathscr{E}$ stable under $O_K$ of degree $\ideal{p}$ so that $\ker(\psi^\ideal{p})=E[\ideal{p}]$. If $\ideal{p}\neq \overline{\ideal{p}}$ then $\mathscr{E}[\ideal{p}]$ and $\mathscr{E}[\overline{\ideal{p}}]$ are the only sub-group schemes of $\mathscr{E}$ of degree $N\ideal{p}$ stable under $O_K$ so that $\ker(\psi^\ideal{p})=E[\ideal{p}]$ or $\ker(\psi^\ideal{p})=E[\overline{\ideal{p}}]$. In the latter case we see that $\psi^\ideal{p}$ is \'etale when restricted to $S_{\overline{\ideal{p}}}\subset S$, which is absurd, hence $\ker(\psi^\ideal{p})=E[\ideal{p}]$.
	\end{proof}
	
	\begin{theo}\label{theo:shimura-is-lambda} In the notation of \textup{(\ref{subsec:cm-over-a-field-shimura-integral})}, the following are equivalent:
		\begin{enumerate}[label=\textup{(\roman*)}]
			\item For each $\ideal{p}\in \mrm{Id}_K^\ideal{g}$ there is a unique homomorphism \[\psi^\ideal{p}: \mathscr{E}\longrightarrow \sigma_\ideal{p}^*(\mathscr{E})\] lifting the $N\ideal{p}$-power relative Frobenius of $\mathscr{E}_\ideal{p}/S_\ideal{p}$ and for each pair of primes $\ideal{p}, \ideal{l}\in \mrm{Id}_K^\ideal{g}$ the diagram \[\xymatrix{\mathscr{E}\ar[r]^{\psi^\ideal{l}}\ar[d]_{\psi^\ideal{p}} & \sigma_\ideal{l}^*(\mathscr{E})\ar[d]^{\sigma_\ideal{l}^*(\psi^\ideal{p})}\\
				\sigma_\ideal{p}^*(\mathscr{E})\ar[r]^{\sigma_\ideal{p}^*(\psi^\ideal{l})} & \sigma_{\ideal{p}\ideal{l}}^*(\mathscr{E})}\] commutes.
			\item $E/L$ is a \textup{CM} elliptic curve of Shimura type.
		\end{enumerate}
	\end{theo}
	\begin{proof} (i) implies (ii): It is enough to show that the action of $G(L^\mrm{ab}/L)$ on $E[\ideal{a}](L^\mrm{ab})$ factors through $G(K^\mrm{ab}/L)$ for all $\ideal{a}$ divisible by $\ideal{g}$. Since the claim takes place only on the generic fibre, we may replace $S=\Spec(O_L[\ideal{g}^{-1}])$ by $\Spec(O_L[\ideal{a}^{-1}])$ and assume that $\ideal{a}=\ideal{g}$.
		
		Then $\mathscr{E}[\ideal{a}]$ is a finite \'etale $S$-scheme and therefore a finite \'etale $\Spec(O_K[\ideal{a}^{-1}])$-scheme. It follows that $\mathscr{E}[\ideal{a}]=\amalg_i \Spec(O_{L_i}[\ideal{a}^{-1}])$ where each $L_i/K$ is a finite extension unramified away from $\ideal{a}$. For each prime ideal $\ideal{p}$ of $O_K$ prime to $\ideal{a}$, the $\Spec(O_K)$-linear morphism $\varphi_\ideal{p}: \mathscr{E}[\ideal{a}]\longrightarrow \mathscr{E}[\ideal{a}]$ induced by the restriction of $\psi_{\mathscr{E}/L}^\ideal{p}$ to $\mathscr{E}[\ideal{a}]$ lifts the absolute $N\ideal{p}$-power Frobenius map of $\mathscr{E}[\ideal{a}]\times_S S_\ideal{p}$ and therefore its restriction to each $\Spec(O_{L_i}[\ideal{a}^{-1}])$ is the Frobenius element corresponding to $\ideal{p}$. As this is true for all primes prime to $\ideal{a}$ it follows that each of the extensions $L_i/K$ is abelian and that $E/L$ is of Shimura type.
		
		(ii) implies (i): We first note that if $E/L$ is a CM elliptic curve of Shimura type then the fixed ideal $\ideal{g}$ cannot be equal to $O_K$. If this were the case then we would have $L=H$ and the elliptic curve $E/H$ would have good reduction everywhere which is impossible by (\ref{prop:no-good-reduction-shimura}). Moreover, the  assumptions and the claims of the theorem are unchanged after replacing $\ideal{g}$ by some power of itself so that as $\ideal{g}\neq O_K$ we may do so in such a way that the reduction map \[O_K^\times\longrightarrow (O_K/\ideal{g})^\times\] is injective.
		
		Now as $E/L$ is of Shimura type it follows that for each $\sigma\in G(L/K)$ we have $\rho_{\sigma^*(E)/L}=\rho_{E/L}$. Therefore, for each prime $\ideal{p}\in \mrm{Id}_K^\ideal{g}$ we have \[(\rho_{\sigma^*_\ideal{p}(E)/L}, c_{\sigma_\ideal{p}^*(E)/L})=(\rho_{E}, c_{\ideal{p}^{-1}\otimes_{O_K}E/L})=(\rho_{\ideal{p}^{-1}\otimes_{O_K}E}, c_{\ideal{p}^{-1}\otimes_{O_K}E/L}).\] In particular, there exists an isomorphism \[f: \ideal{p}^{-1}\otimes_{O_K} E\isomto \sigma_\ideal{p}^*(E)\] whose extension to the N\'eron models (relative to $S$) we again denote by $f$.
		
		For each prime $\ideal{P}$ of $O_L$ lying over $\ideal{p}$ there exists a unique element $\epsilon_\ideal{P}\in O_K^\times$ such that \[f_\ideal{P}:=\epsilon_\ideal{P} f: \ideal{p}^{-1}\otimes_{O_K} \mathscr{E}\isomto \sigma_\ideal{p}^*(\mathscr{E})\] reduces modulo $\ideal{P}$ to the unique isomorphism \[\ideal{p}^{-1}\otimes_{O_K} \mathscr{E}_\ideal{P}\isomto \mrm{Fr}_{S_\ideal{P}}^{N\ideal{p}*}(\mathscr{E}_\ideal{P})\] whose composition with $i_\ideal{p}: \mathscr{E}\longrightarrow \ideal{p}^{-1}\otimes_{O_K}\mathscr{E}$ is the $N\ideal{p}$-power relative Frobenius. We write $\psi^\ideal{P}=i_\ideal{p}\circ f_\ideal{P}: \mathscr{E}\to \sigma_\ideal{p}^*(\mathscr{E})$ for this composition.
		
		Since $E/L$ is an elliptic curve of Shimura type \[\mathscr{E}[\ideal{g}]=\amalg_{i} \Spec(O_{L_i}[\ideal{g}^{-1}])\] where each $L_i/K$ is an abelian extension, unramified away from $\ideal{g}$. Let $\varphi^{\ideal{p}}: \mathscr{E}[\ideal{g}]\to\sigma_\ideal{p}^*(\mathscr{E}[\ideal{g}])$ be the sum of the Frobenius elements at $\ideal{p}$ of the extensions $L_i/L$.
		
		As $\Spec(O_L[\ideal{g}^{-1}])$ is connected, the reduction map \[\Hom_{O_L[\ideal{g}^{-1}]}(\mathscr{E}[\ideal{g}], \sigma_\ideal{p}^*(\mathscr{E})[\ideal{g}])\to \Hom_{O_L/\ideal{P}}(\mathscr{E}_\ideal{P}[\ideal{g}], \sigma_\ideal{p}^*(\mathscr{E})_\ideal{P}[\ideal{g}])\] is injective and by definition the images of $\psi^\ideal{P}|_{\mathscr{E}[\ideal{g}]}$ and $\varphi^\ideal{p}$ coincide. Hence, $\psi^\ideal{P}|_{\mathscr{E}[\ideal{g}]}$ depends only on $\ideal{p}$. However, for $\ideal{P}, \ideal{P}'$ each dividing $\ideal{p}$, we have $\psi^{\ideal{P}'}=\epsilon \psi^{\ideal{P}}$ for some $\epsilon \in O_K^\times.$ As \[\psi^\ideal{P}|_{\mathscr{E}[\ideal{g}]}=\psi^{\ideal{P}'}|_{\mathscr{E}[\ideal{g}]}=\epsilon\psi^{\ideal{P}}|_{\mathscr{E}[\ideal{g}]}\] and $O_K^\times \to (O_K/\ideal{g})^\times$ is injective, it follows that $\epsilon=1$ and that $\psi^{\ideal{P}}$ depends only on $\ideal{p}.$ We write \[\psi^{\ideal{p}} : \mathscr{E}\longrightarrow \sigma_\ideal{p}^*(\mathscr{E})\] for this common value, which lifts the $N\ideal{p}$-power Frobenius modulo $\ideal{p}$ by construction.
		
		For a pair of prime ideals $\ideal{p}$, $\ideal{l}\in \mrm{Id}_K^\ideal{g}$ consider the diagram \[\xymatrix{\mathscr{E}\ar[r]^{\psi^\ideal{l}}\ar[d]_{\psi^\ideal{p}} & \sigma_\ideal{l}^*(\mathscr{E})\ar[d]^{\sigma_\ideal{l}^*(\psi^\ideal{p})}\\
			\sigma_\ideal{p}^*(\mathscr{E})\ar[r]^{\sigma_\ideal{p}^*(\psi^\ideal{l})} & \sigma_{\ideal{p}\ideal{l}}^*(\mathscr{E})}\] and let $h: \mathscr{E}\longrightarrow \sigma_{\ideal{p}\ideal{l}}^*(\mathscr{E})$ be the difference of the two compositions from top left to bottom right. The diagram commutes when restricted to the $\ideal{a}$-torsion for any ideal $\ideal{a}$ prime to $\ideal{g}$, as both compositions induce the `Frobenius element' corresponding to $\ideal{p}\ideal{l}$ of the finite \'etale $S$-schemes $\mathscr{E}[\ideal{a}]$. Therefore $\mathscr{E}[\ideal{a}]\subset \ker(h)$ for all $\ideal{a}$ prime to $\ideal{g}$ and this is only possible if $\ker(h)=\mathscr{E}$ so that $h=0$ and the diagram commutes.
	\end{proof}
	
	\section{Minimal models}
	
	\subsection{} \label{subsec:lambda-module-for-cm-curves} We now consider the existence of certain global minimal models of elliptic curves of Shimura type. First, we consider some consequences of the existence of commuting families of Frobenius lifts.
	
	We continue with the notation of (\ref{subsec:cm-over-a-field-shimura-integral}) but will also assume that $E/L$ is of Shimura type. Thus for each $\ideal{p}\nmid \ideal{g}$ there is a (unique) isomorphism \[\nu_\ideal{p}: \ideal{p}^{-1}\otimes_{O_K} \mathscr{E}\longrightarrow \sigma^*(\mathscr{E})\] with the property that $\nu_\ideal{p}\circ i_\ideal{p}=\psi^\ideal{p}: \mathscr{E}\longrightarrow \sigma_\ideal{p}^*(\mathscr{E})$ lifts the $N\ideal{p}$-power relative Frobenius. For a pair of primes $\ideal{p}$, $\ideal{l}$ prime to $\ideal{g}$ semi-commutativity of the isogenies $\psi^{\ideal{p}}$ and $\psi^{\ideal{l}}$ (\ref{theo:shimura-is-lambda}) expressed in terms of the isomorphisms $\nu_{\ideal{l}}$ and $\nu_\ideal{p}$ becomes: \begin{equation}\sigma_{\ideal{l}}^*(\nu_\ideal{p})\circ (\ideal{p}^{-1}\otimes_{O_K}\nu_{\ideal{l}})=\sigma_{\ideal{p}}^*(\nu_{\ideal{l}})\circ (\ideal{l}^{-1}\otimes_{O_K}\nu_{\ideal{p}}).\label{eqn:frob-commute-prime}\end{equation} For any ideal $\ideal{a}\in \mrm{Id}_{O_K}^\ideal{g}$, choosing a prime factorisation of $\ideal{a}$, we may define isomorphisms \[\nu_\ideal{a}: \ideal{a}^{-1}\otimes_{O_K}\mathscr{E}\isomto \sigma_\ideal{a}^*(\mathscr{E})\] by composing the $\nu_\ideal{p}$ the appropriate number of times for for $\ideal{p}|\ideal{a}$. The resulting isomorphism $\nu_\ideal{a}$ is independent of the order of the composition by virtue of (\ref{eqn:frob-commute-prime}) and for any pair of ideals $\ideal{a}$ and $\ideal{b}$ prime to $\ideal{g}$ they satisfy: \begin{equation}\sigma_{\ideal{a}}^*(\nu_\ideal{b})\circ (\ideal{b}^{-1}\otimes_{O_K}\nu_{\ideal{a}})=\sigma_{\ideal{b}}^*(\nu_{\ideal{a}})\circ (\ideal{a}^{-1}\otimes_{O_K}\nu_{\ideal{b}}).\label{eqn:frob-commute-all}\end{equation}
	
	\subsection{} If $\ideal{a}$ is an ideal prime to $\ideal{g}$ such that $\sigma_{\ideal{a}}=\mathrm{id}_L$ then $\nu_{\ideal{a}}$ is an isomorphism \[\nu_\ideal{a}: \ideal{a}^{-1}\otimes_{O_K}\mathscr{E}\longrightarrow \mathscr{E}\] and so must be of the form $l(\ideal{a})\otimes \mrm{id}_{\mathscr{E}}$ where $l(\ideal{a})\in O_K$ is a generator of $\ideal{a}$. Thus if $P_{L/K}^{\ideal{g}}$ denotes the monoid of ideals $\ideal{a}$ of $O_K$ which are prime to $\ideal{g}$ and which satisfy $\sigma_{\ideal{a}}=\mathrm{id}_L\in G(L/K)$ we obtain a multiplicative map \[l:P_{L/K}^{\ideal{g}}\longrightarrow O_K: \ideal{a}\mapsto l(\ideal{a})\] satisfying $l(\ideal{a})\cdot O_K=\ideal{a}$.
	
	If $\ideal{f}$ is an ideal of $O_K$ with the property that the group scheme $E[\ideal{f}]$ is constant then $\mathscr{E}[\ideal{f}]$ is also constant and the composition \[\mathscr{E}[\ideal{f}]\longrightarrow \ideal{a}^{-1}\otimes_{O_K}\mathscr{E}[\ideal{f}]\stackrel{\nu_\ideal{a}}{\longrightarrow} \mathscr{E}[\ideal{f}]\] is multiplication by $l(\ideal{a})$. However, it is also equal to the sum of the Frobenius elements of the connected components of $\mathscr{E}[\ideal{f}]$ which, as $\mathscr{E}[\ideal{f}]$ is constant, must be the identity. Therefore if $E[\ideal{f}]$ is constant, then $l$ satisfies $l(\ideal{a})=1\bmod \ideal{f}$.
	
	\subsection{} By the N\'eron mapping property the isomorphisms $\nu_\ideal{a}$ extend to isomorphisms on the full N\'eron model over $\Spec(O_L)$ (which is no longer an elliptic curve, only a smooth one dimensional group scheme) \[\nu_\ideal{a}: \ideal{a}^{-1}\otimes_{O_K}\mrm{Ner}_{O_L}(E)=\mrm{Ner}_{O_L}( \ideal{a}^{-1}\otimes_{O_K}E)\isomto \sigma_\ideal{a}^*(\mrm{Ner}_{O_L}(E))\] satisfying the same commutativity condition. Writing \[T=\underline{\Lie}_{\mrm{Ner}_{O_L}(E)/O_L}\] for the Lie algebra of the N\'eron model, which is a projective rank one $O_L$-module, the isomorphisms $\nu_\ideal{a}$ induce $O_L$-isomorphisms (which we denote by the same letter) \[\nu_\ideal{a}: \ideal{a}^{-1}\otimes_{O_K}T\isomto \sigma_\ideal{a}^*(T)\] for each $\ideal{a}\in \mrm{Id}_{O_K}^\ideal{g}$. These satisfy the same commutativity condition (\ref{eqn:frob-commute-all}) as the (original) $\nu_\ideal{a}$ and moreover if $\ideal{a}\in P_{L/K}^\ideal{g}$ then $\nu_\ideal{a}=l(\ideal{a})\otimes_{O_K}\mathrm{id}_T.$
	
	\begin{theo}\label{theo:shimura-curves-minimal-model} In the notation of \textup{(\ref{subsec:lambda-module-for-cm-curves})}, if $L=K(\ideal{f})$ is a ray class field and $E[\ideal{f}]$ is constant then $T\otimes_{O_K}O_K[\ideal{f}^{-1}]$ is free. In other words, $E/\Spec(K(\ideal{f}))$ admits a global minimal model away from $\ideal{f}$.
	\end{theo}
	\begin{proof} We apply (\ref{prop:tannaka-result-general}) to extend the map $l: P^\ideal{g}_{L/K}\longrightarrow O_K$ to a map \[l:\mrm{Id}_{O_K}^{\ideal{g}}\longrightarrow O_{K(\ideal{f})}\] satisfying \begin{equation} l(\ideal{a})\cdot O_{K(\ideal{f})}=\ideal{a}\cdot O_{K(\ideal{f})} \quad \text{ and } \quad l(\ideal{a}\ideal{b})=l(\ideal{a})\sigma_\ideal{a}(l(\ideal{b}))\label{eqn:l-com-condition}\end{equation} for all $\ideal{a}, \ideal{b}\in \mrm{Id}_K^{\ideal{g}}$.
		
	We then define, for each $\ideal{a}\in \mathrm{Id}_{K}^\ideal{g}$, an isomorphism $t_\ideal{a}: T\longrightarrow \sigma_\ideal{a}^*(T)$ by \[T\stackrel{l(\ideal{a})^{-1}\otimes \mrm{id}}{\longrightarrow} \ideal{a}^{-1}\otimes_{O_K}T\stackrel{\nu_\ideal{a}}{\longrightarrow} \sigma^*_a(T).\] If $\ideal{a}\in \mrm{P}_{L/K}^{\ideal{g}}$ then by (\ref{eqn:l-com-condition}) we have \[t_\ideal{a}=l(\ideal{a})^{-1}\otimes l(\ideal{a})=\mrm{id}_T.\] This, combined with the commutativity conditions (\ref{eqn:frob-commute-all}) on the $\nu_\ideal{a}$ and (\ref{eqn:l-com-condition}) on the $l(\ideal{a})$, shows that $t_\ideal{a}$ depends only on the class $\sigma_\ideal{a}\in G(K(\ideal{f})/K)$, so that we may instead write $t_{\ideal{a}}=t_{\sigma_\ideal{a}}$.
		
	We now have a collection of isomorphisms $t_\sigma: T\isomto \sigma^*(T)$ indexed by $\sigma\in G(K(\ideal{f})/K)$ which satisfy: \[t_{\mrm{id}_{K(\ideal{f})}}=\mrm{id}_T \quad \text{ and } \quad t_{\sigma\tau}=t_\sigma\circ \sigma^*(t_\tau)\] for $\sigma, \tau\in G(K(\ideal{f})/K)$. In other words, the isomorphisms $t_\sigma$ define Galois descent data on $T$ relative to $O_{K}\longrightarrow O_{K(\ideal{f})}$. The homomorphism $O_K\longrightarrow O_{K(\ideal{f})}$ is finite and \'etale after inverting $\ideal{f}$ and after doing so the isomorphisms $t_\sigma$ define actual decent data relative to $O_{K}[\ideal{f}^{-1}]\longrightarrow O_{K(\ideal{f})}[\ideal{f}^{-1}]$. Therefore there exists an $O_K[\ideal{f}^{-1}]$-module $T_0$ such that \[T_0\otimes_{O_K[\ideal{f}^{-1}]} O_{K(\ideal{f})}[\ideal{f}^{-1}]\isomto T\otimes_{O_{K(\ideal{f})}}O_{K(\ideal{f})}[\ideal{f}^{-1}].\] However, as $K(\ideal{f})$ contains the Hilbert class field $H=K(1)$, all rank one projective $O_{K}[\ideal{f}^{-1}]$-modules become free after base change to $O_{K(\ideal{f})}[\ideal{f}^{-1}]$, it follows that \[T_0\otimes_{O_K}O_{K}[\ideal{f}^{-1}]\isomto T\otimes_{O_K}O_{K}[\ideal{f}^{-1}]\] is free.
	\end{proof}
	
	The above result is in fact really only interesting when $\ideal{f}$ is small. In fact, if $\ideal{f}$ has the property that $O_K^\times\longrightarrow (O_K/\ideal{f})^\times$ is injective then there is (up to automorphisms of $K(\ideal{f})$) only one CM elliptic curve over $K(\ideal{f})$ and by (\ref{coro:good-red-const}) it has good reduction everywhere. Thus we get a CM elliptic curve $\mathscr{E}/\Spec(O_{K(\ideal{f})})$ which is nothing more than the universal CM elliptic curve with level-$\ideal{f}$ structure.\footnote{$O_{K(\ideal{f})}$ being the moduli stack of such CM elliptic curves.} However, if $\ideal{f}=(1)$ we obtain the following corollary which is a strengthening of a result of Gross (Corollary 4.4 of \cite{Gross82}) who proved it for elliptic curves of Shimura type\footnote{Technically, Gross' result is for CM elliptic curves over defined over the Hilbert class field of an imaginary quadratic with prime discriminant whose Hecke character is $G(H/K)$-invariant. However, these conditions actually imply that $E/H$ is a CM elliptic curve of Shimura type. Indeed, the primality of the discriminant implies that the class number of $K$ is prime to the order of $O_K^\times$, which combined with the $G(H/K)$-invariance of the Hecke character implies that $E/H$ is of Shimura type. For a result along these lines see Proposition 2 of \cite{Gilles85}.} over $H$ in the case where $K$ has prime discriminant:
	
	\begin{coro} If $E/H$ is an elliptic curve of Shimura type then $E$ admits a global minimal model.
	\end{coro}
	\begin{proof} This is just (\ref{theo:shimura-curves-minimal-model}) with $\ideal{f}=(1)$, noting that $E[1]=\Spec(H)$ is always constant.
	\end{proof}

	\appendix
	
	\section{A principal ideal theorem} The goal of this appendix is to prove a principal ideal theorem for number fields. Before explaining the result, let us first introduce some notation.
	
	\subsection{} Let $K$ be a number field with ring of integers $O_K$ and let $\ideal{f}$ be an ideal of $O_K$. If $a\in K^\times$ we write $a=1\bmod \ideal{f}$ to mean that $\ideal{f}$ divides the principal ideal $(a-1)$. For a pair of fractional ideals $\ideal{a}, \ideal{b}$ of $K$ we write $\ideal{a}=\ideal{b}\bmod \ideal{f}$ to mean that $\ideal{a}\ideal{b}^{-1}=(a)$ with $a=1\bmod \ideal{f}$.
	
	Now if $L/K$ is a finite extension, totally split at all infinite places of $K$, we denote by $\ideal{f}_{L/K}$, $\ideal{D}_{L/K}$ and $\ideal{F}_{L/K}$ the conductor, different and genus ideal of the extension $L/K$ (see \S 1 of \cite{Terada1952} for precise definitions). These are integral ideals of $K$, $L$ and $K$ respectively and we have \[\ideal{f}_{L/K}=\ideal{D}_{L/K}\ideal{F}_{L/K}.\] Given an intermediate extension $L/K'/K$ we also define \[\ideal{f}_{L/K'/K}=\ideal{D}_{L/K'}\ideal{F}_{L/K}.\] While a priori is an ideal of $L$, is in fact an ideal of $K'$ (see \cite{Tannaka58}). We also note that the ideals $\ideal{f}_{L/K}$, $\ideal{F}_{L/K}$, and $\ideal{f}_{L/K'/K}$ are all invariant under $G(L/K)$.
	
	We will be interested only in the case where $L=K(\ideal{f})$ is the ray class field of conductor $\ideal{f}$. Note while the conductor of $K(\ideal{f})/K$ always divides $\ideal{f}$ it may do so strictly. However, if $L/K$ is a (wide) ray class field and its conductor is $\ideal{f}_{L/K}$ then $K(\ideal{f}_{L/K})=L$.
	
	For other notation we refer the reader to the section following the introduction.
	
	\subsection{} In \cite{Tannaka58} Tannaka proved the following:
	\begin{theo*}[Tannaka] Let $K$ be a number field and $L/K$ be a ray class field with conductor $\ideal{f}_{L/K}$ and genus ideal $\ideal{F}_{L/K}$. Then there exist elements $\Theta(\ideal{a})\in L^\times$, indexed by ideals $\ideal{a}\in \mrm{Id}_K^{\ideal{f}_{L/K}}$ such that: \begin{enumerate}[label=\textup{(\roman*)}]
			\item $\Theta(\ideal{a})\cdot O_L=\ideal{a}\cdot O_L$,
			\item $\Theta(\ideal{a})=1\bmod \ideal{F}_{L/K}$, and
			\item $\Theta(\ideal{a})\sigma_\ideal{a}(\Theta(\ideal{b}))=\Theta(\ideal{a}\ideal{b}) \bmod O_K^\times.$
		\end{enumerate}
	\end{theo*}

In this appendix we will prove the following strengthening of Tannaka's Theorem:
	\begin{theo}\label{prop:tannaka-result-general} Let $K$ be a number field, let $L/K$ be a ray class field with conductor $\ideal{f}=\ideal{f}_{L/K}$, let $\ideal{g}$ be an ideal divisible by $\ideal{f}$ and let \[l: \mrm{P}^{\ideal{g}, \ideal{f}}_K\longrightarrow K^\times: \ideal{a}\mapsto l(\ideal{a})\] be a homomorphism such that $l(\ideal{a})\cdot O_K=\ideal{a}\subset K$ and such that $l(\ideal{a})=1\bmod \ideal{f}_{L/K}$. Then $l$ can be extended to a map \[l: \mrm{Id}^{\ideal{g}}_K\longrightarrow L^\times: \ideal{a} \mapsto l(\ideal{a})\] such that:
		\begin{enumerate}[label=\textup{(\roman*)}]
			\item $l(\ideal{a})\cdot O_{L}=\ideal{a}\cdot O_{L}$,
			\item $l(\ideal{a})=1\bmod \ideal{F}_{L/K}$ and
			\item for all $\ideal{a}, \ideal{b}\in \mrm{Id}^{\ideal{g}}_K$ we have $l(\ideal{a}\ideal{b})=l(\ideal{a})\sigma_{\ideal{a}}(l(\ideal{b})).$
		\end{enumerate}
	\end{theo}
	
	Despite its different form, this is indeed a strengthening of Tannaka's Theorem. As $\mrm{P}_K^{\ideal{f}_{L/K}, \ideal{f}_{L/K}}\subset \mrm{Id}_{K}$ and $\mrm{Id}_{K}$ is a free abelian group it follows that $\mrm{P}_K^{\ideal{f}_{L/K}, \ideal{f}_{L/K}}$ is free abelian and so (choosing appropriate generators) there always exist multiplicative functions \[l: \mrm{P}_K^{\ideal{f}_{L/K}, \ideal{f}_{L/K}}\longrightarrow K^\times\] satisfying $l(\ideal{a})\cdot O_K=\ideal{a}$ and $l(\ideal{a})=1 \bmod \ideal{f}_{L/K}$. Therefore, applying (\ref{prop:tannaka-result-general}) to the map $l$ we obtain a map $l: \mrm{Id}_{K}^{\ideal{f}_{L/K}}\longrightarrow L^\times$ and setting $\Theta(\ideal{a})=l(\ideal{a})$ for each $\ideal{a}\in \mrm{Id}_{O_K}^{\ideal{f}_{L/K}}$, we may replace (iii) of Tannaka's Theorem by: $\Theta(\ideal{a})\sigma_\ideal{a}(\Theta(\ideal{b}))=\Theta(\ideal{a}\ideal{b})$
	
	While this result is new, the methods used to prove it follow closely those of Tannaka and main innovation consists of applying several classical results in class field theory before invoking Tannaka's (very technical) proof. In order to ease comparison with Tannaka's proof, we have adopted identical notation to \cite{Tannaka58}, save that we do not denote the action of elements of Galois groups exponentially, writing $\sigma(a)$ for Tannaka's $a^{\sigma}$.
	
	\subsection{} We now recall a selection of results which we will need during the proof of (\ref{prop:tannaka-result-general}).
	
	\begin{theo}[Hasse's Norm Theorem] Let $L/K$ be a finite cyclic extension. Then $N_{L/K}(I_L)\cap K^\times=N_{L/K}(L^\times)$.
	\end{theo}
	
	\begin{theo}[Noether's `Theorem 90'] Let $L/K$ be a finite cyclic extension with generator $\sigma\in G(L/K)$ and let $\ideal{a}$ be a fractional ideal of $L$. Then $N_{L/K}(\ideal{a})=O_K$ if and only if $\ideal{a}=\ideal{b}\sigma(\ideal{b})^{-1}$ for some fractional ideal $\ideal{b}$ of $L$.
	\end{theo}
	\begin{proof} As $G(L/K)$ is cyclic and generated by $\sigma$ we have \[\frac{\{\ideal{a}\in \mrm{Id}_{L}: N_{L/K}(\ideal{a})=O_L\}}{\{\ideal{a}\sigma(\ideal{a})^{-1}: \ideal{a}\in \mrm{Id}_L\}}\isomto H^1(G(L/K), \mrm{Id}_{L}): \ideal{a} \mapsto (\sigma^i\mapsto \ideal{a} \sigma(\ideal{a})\cdots\sigma^{i-1}(\ideal{a})).\] By Proposition 6, \S 13, Chapter V of \cite{Lemmermeyer07} the group $H^1(G(L/K), \mrm{Id}_{L})$ vanishes which is the claim.
	\end{proof}
	
	\begin{theo}[Terada's Norm Theorem] Let $L/K$ be a finite cyclic extension with generator $\sigma\in G(L/K)$, let $\alpha\in L^\times$ and let $\ideal{m}$ be an ideal of $O_L$. Then the following are equivalent:
		\begin{enumerate}[label=\textup{(\roman*)}]
			\item $N_{L/K}(\alpha)=1\bmod \ideal{f}_{L/K}\ideal{m}$.
			\item $\alpha=\beta\sigma(\beta)^{-1}\bmod \ideal{F}_{L/K}\ideal{m}$ for some $\beta\in L$.
		\end{enumerate}
	\end{theo}
	\begin{proof} This is Theorem 2 of \cite{Terada1952}.
	\end{proof}
	
	Finally, with notation as in (\ref{prop:tannaka-result-general}), let $\ideal{p}_1, \ldots, \ideal{p}_r$ be prime ideals of $K$, prime to $\ideal{g}$ and with the property that \[G(L/K)=\bigoplus_{i=1}^r \langle \sigma_i \rangle\] where $\sigma_i=\sigma_{\ideal{p}_i}$. Also, for $1\leq i\leq n$ let $n_i$ be the order of $\sigma_i$ and let $K_i\subset L$ be the sub-extension fixed by the sub-group of $G(L/K)$ generated by $\{\sigma_j\}_{j\neq i}$, so that $G(K_i/K)\isomto \langle \sigma_i \rangle$.
	
	\begin{theo}\label{thm:tannaka-principal} For $1\leq i\leq r$ let $\ideal{a}_i$ be a fractional ideal of $K_i$ such that $\ideal{p}_i=\ideal{a}_i\sigma_i(\ideal{a}_i)^{-1}\bmod \ideal{f}_{L/K_i/K}$. Then $\ideal{a}_1\cdots \ideal{a}_r=O_L \bmod \ideal{F}_{L/K}$.
	\end{theo}
	\begin{proof} This is Theorem 3 of \cite{Tannaka58}.
	\end{proof}
	
	\subsection{} We are now proceed with the proof of (\ref{prop:tannaka-result-general}) and shall also use the notation introduced above preceding (\ref{thm:tannaka-principal}).
	
	For each $1\leq i\leq r$ let us make the following constructions. As $\ideal{p}_i^{n_i}\in \mrm{P}_K^{\ideal{g}, \ideal{f}_{L/K}}$ we have  $l(\ideal{p}_i^{n_i})=1\bmod \ideal{f}_{L/K}$ and a fortiori $l(\ideal{p}_i^{n_i})=1\bmod \ideal{f}_{K_i/K}$ so that $l(\ideal{p}_i^{n_i})\in N_{K_i/K}(I_{K_i})$ (i.e. $l(\ideal{p}_i^{n_i})$ is a local norm everywhere). However, $l(\ideal{p}_i^{n_i})$ is also an element of $K$ so that by Hasse's Norm Theorem, there is some $\pi_i\in K_i$ with $N_{K_i/K}(\pi_i)=l(\ideal{p}_i^{n_i})$. By construction we have \[N_{K_i/K}(\ideal{p}_i(\pi_i)^{-1})=\ideal{p}_i^{n_i}\ideal{p}_i^{-n_i}=O_K\] so that, by Noether's `Theorem 90', there is an ideal $\ideal{b}_i$ of $K_i$ with \[\ideal{p}_i(\pi_i)^{-1}=\ideal{b}_i\sigma_i(\ideal{b}_i)^{-1}.\] We now apply Terada's Norm Theorem (with $\ideal{m}=\ideal{f}_{L/K}\ideal{f}_{K_i/K}^{-1}$) to find $\alpha_i, \beta_i\in K_i$ satisfying $\alpha_i=1\bmod \ideal{f}_{L/K_i/K}$, $\pi_i =\alpha_i \beta_i\sigma_i(\beta_i)^{-1}$ and \begin{equation}\label{eqn:def-alpha} N_{K_i/K}(\pi_i)=N_{K_i/K}(\alpha_i \beta_i\sigma_i(\beta_i)^{-1})=N_{K_i/K}(\alpha_i)=l(\ideal{p}_i^{n_i}).
	\end{equation} Finally, we set $\ideal{a}_i=(\beta_i)\ideal{b}_i$ so that $\ideal{p}_i=(\alpha_i) \ideal{a}\sigma_i^{-1}(\ideal{a}_i)$ and hence \[\ideal{p}_i=\ideal{a}\sigma_i^{-1}(\ideal{a}_i)\bmod \ideal{f}_{L/K_i/K}.\] The ideals $\ideal{a}_i$ for $1\leq i\leq r$ satisfy the conditions of (\ref{thm:tannaka-principal}) so that there is an $A\in L^\times$ with $A=1\bmod \ideal{F}_{L/K}$ and \[\ideal{a}_1\cdots \ideal{a}_r=(A).\] Finally, we set \[\Theta_i=\alpha_i A\sigma_i(A)^{-1}\] and note that $\Theta_i\cdot O_L=\ideal{p}_i\cdot O_L$.
	
	We now go about extending the map $l$ and doing so following rather closely the method of \S 1 of \cite{Tannaka58}. Each $\ideal{a}\in \mrm{Id}_K^{\ideal{g}}$ can be written uniquely as a product of ideals \[\ideal{a}=\gamma(\ideal{a})\cdot \prod_{i=1}^r \ideal{p}_i^{x_i}\] with $\gamma(\ideal{a})\in \mrm{P}_K^{\ideal{g},\ideal{f}_{L/K}}$ and $0\leq x_i<n_i$. Before we continue, it will be useful to note the values of $\gamma(\ideal{a}\ideal{b})$ for a few specific choices of the ideal $\ideal{b}$:
	\begin{enumerate}[label=\textup{(\roman*)}]
		\item If $\ideal{b}\in \mrm{P}_K^{\ideal{g}, \ideal{f}_{L/K}}$ then \begin{equation}\label{eqn:mult-prin} \gamma(\ideal{a}\ideal{b}) = \gamma(\ideal{a})\gamma(\ideal{b}).
		\end{equation}
		\item If $\ideal{b}=\ideal{p}_j$ and $x_j\neq n_j-1$ then \begin{equation}\label{eqn:mult-prime-triv} \gamma(\ideal{p}_j)=O_K \quad \text{ and } \quad \gamma(\ideal{a}\ideal{p}_j)=\gamma(\ideal{a}).
		\end{equation}
		\item If $\ideal{b}=\ideal{p}_j$ and $x_j=n_j-1$ then \begin{equation}\label{eqn:mult-prime-carry}\gamma(\ideal{a}\ideal{p}_j)=\gamma(\ideal{a})\gamma(\ideal{p}_j^{n_j}). 
		\end{equation}
	\end{enumerate}
	Still following \S 1 of \cite{Tannaka58} we now define $l(\ideal{a})$ for any $\ideal{a}\in \mrm{Id}_{K}^{\ideal{g}}$ by \[l(\ideal{a}):=l(\gamma(\ideal{a}))\prod_{i=1}^n w_{i, x_i}(\Theta_i)\]
	where \[w_{i, x_i}=\left(\sum_{j=1}^{x_i}\sigma_{i}^j\right)\cdot\prod_{k=1}^{i-1} \sigma_{k}^{x_k}\in \Z[G(L/K)]\] and where $\Z[G(L/K)]$ acts multiplicatively on elements of $L^\times$: $(\sigma+\tau)a=\sigma(a)\tau(a)$. It is clear that $l(\ideal{a})\cdot O_{L}=\ideal{a}\cdot O_{L}$ and that this map does indeed extend the given map $l$. Moreover, by construction we have the relations \[A=1\bmod \ideal{F}_{L/K} \quad\quad l(\ideal{a})=1\bmod \ideal{f}_{L/K} \quad\quad \alpha_i=1\bmod \ideal{f}_{L/K_i/K},\] so that as $\ideal{F}_{L/K}$ divides both $\ideal{f}_{L/K}$ and $\ideal{f}_{L/K_i/K}$, and as $\ideal{F}_{L/K}$, $\ideal{f}_{L/K}$ and $\ideal{f}_{L/K_i/K}$ are invariant under the action of $G(L/K)$, we get the relation \[l(\ideal{a})=1\bmod \ideal{F}_{L/K}.\] All that remains to be shown is that $l(\ideal{a}\ideal{b})=l(\ideal{a})\sigma_\ideal{a}(l(\ideal{b}))$ for $\ideal{a}, \ideal{b}\in \mrm{Id}_K^{\ideal{g}}$.
	
	So let $\ideal{b}\in \mrm{Id}_K^{\ideal{g}}$ be another fractional ideal, and also write \[\ideal{b}=\gamma(\ideal{b})\cdot \prod_{i=1}^r \ideal{p}_i^{y_i} \quad \text{ and }\quad \ideal{a}\ideal{b}=\gamma(\ideal{a}\ideal{b})\cdot \prod_{i=1}^r \ideal{p}_i^{z_i}\] where $\gamma(\ideal{b}), \gamma(\ideal{a}\ideal{b})\in \mrm{P}_K^{\ideal{g}, \ideal{f}_{L/K}}$ and $0\leq y_i, z_i<n_i$. Define $\delta_i\in \{0, 1\}$ for $1\leq i\leq r$ by the equality \[z_i=x_i+y_i-\delta_i n_i.\] Then one finds (see equation (9) of \cite{Tannaka58}) \begin{equation}\label{eqn:tannaka}\frac{l(\ideal{a})\sigma_\ideal{a}(l(\ideal{b}))}{l(\ideal{a}\ideal{b})}=\frac{l(\gamma(\ideal{a}))l(\gamma(\ideal{b}))}{l(\gamma(\ideal{a}\ideal{b}))}\cdot \prod_{i=1}^r N_{K_i/K}(\alpha_i)^{\delta_i}.\end{equation}
	
	It remains to check that the right hand side of (\ref{eqn:tannaka}) is equal to one for all ideals $\ideal{b}$. The group $\mrm{Id}_{K}^{\ideal{g}}$ is generated (as a monoid!) by $\mrm{P}_K^{\ideal{g}, \ideal{f}_{L/K}}$ and $\ideal{p}_1, \ldots, \ideal{p}_n$ so that by induction, it is enough to check that $l(\ideal{a}\ideal{b})=l(\ideal{a})\sigma_\ideal{a}(l(\ideal{b}))$ when $\ideal{b}\in \mrm{P}_K^{\ideal{g}, \ideal{f}_{L/K}}$ or when $\ideal{b}=\ideal{p}_j$ for $1\leq j\leq r.$
	If $\ideal{b}\in \mrm{P}_K^{\ideal{g}, \ideal{f}_{L/K}}$ then $\delta_i=0$ for $1\leq i\leq r$ and by (\ref{eqn:mult-prin}) we have $\gamma(\ideal{a})\gamma(\ideal{b})=\gamma(\ideal{a}\ideal{b})$ so that \[\frac{l(\gamma(\ideal{a}))l(\gamma(\ideal{b}))}{l(\gamma(\ideal{a}\ideal{b}))}\prod_{i=1}^r N_{K_i/K}(\alpha_i)^{\delta_i}=1.\]
	
	If $\ideal{b}=\ideal{p}_j$ and $x_j\neq n_j-1$ then $\delta_i=0$ for $1\leq i\leq r$ and by (\ref{eqn:mult-prime-triv}) we have $\gamma(\ideal{p}_j)=O_K$ and $\gamma(\ideal{a}\ideal{p}_j)=\gamma(\ideal{a})$ so that \[\frac{l(\gamma(\ideal{a}))l(\gamma(\ideal{b}))}{l(\gamma(\ideal{a}\ideal{b}))}\prod_{i=1}^r N_{K_i/K}(\alpha_i)^{\delta_i}=1.\]
	
	Finally, if $\ideal{b}=\ideal{p}_j$ and $x_j=n_j-1$ then $\delta_i=0$, unless $i=j$ in which case $\delta_j=1$, so that \[\prod_{i=1}^r N_{K_i/K}(\alpha_i)^{\delta_i}=N_{K_j/K}(\alpha_j)=l(\gamma(\ideal{p}^{n_j}_j))\] by (\ref{eqn:def-alpha}). By (\ref{eqn:mult-prime-carry}) we have $\gamma(\ideal{a}\ideal{p}_j)=\gamma(\ideal{a})\gamma(\ideal{p}_j^{n_j})$ so that \[\frac{l(\gamma(\ideal{a}))l(\gamma(\ideal{b}))}{l(\gamma(\ideal{a}\ideal{b}))}\prod_{i=1}^r N_{K_i/K}(\alpha_i)^{\delta_i}=l(\gamma(\ideal{p}_j^{n_j}))^{-1}l(\gamma(\ideal{p}_j^{n_j}))=1.\] Therefore, for all $\ideal{a}, \ideal{b}\in \mrm{Id}_K^{\ideal{g}}$ we have $l(\ideal{a}\ideal{b})=l(\ideal{a})\sigma_\ideal{a}(l(\ideal{b}))$ and (\ref{prop:tannaka-result-general}) is proven.
	\bibliographystyle{alpha}
	\bibliography{shim}
\end{document}